\documentclass[12pt,a4paper]{amsart}
\usepackage{mathrsfs}
\usepackage{amsfonts}
\usepackage{txfonts}

\usepackage{hyperref}
\usepackage{latexsym}
\usepackage{amssymb}
\allowdisplaybreaks

\newtheorem{theorem}{Theorem}[section]
\newtheorem{lemma}[theorem]{Lemma}
\newtheorem{corollary}[theorem]{Corollary}
\newtheorem{proposition}[theorem]{Proposition}
\newtheorem{example}[theorem]{Example}

\theoremstyle{definition}
\newtheorem{definition}[theorem]{Definition}

\newtheorem{remark}{Remark}[section]

\numberwithin{equation}{section}

\begin{document}

\title[On $\alpha$-monotone operators and their resolvent in Banach spaces]
{{\bf On $\alpha$-monotone operators and their resolvent in Banach spaces}}

\author{ Changchi Huang, Jigen Peng, Yuchao Tang$^*$ }

\address{ Changchi Huang:  School of Mathematics And Information Science, Guangzhou University,
Guangzhou, 510006, China}
\email{cchuang@gzhu.edu.cn\;\;(C. Huang)}
\address{ Jigen Peng:   School of Mathematics And Information Science, Guangzhou University,
Guangzhou, 510006, China}
\email{jgpeng@gzhu.edu.cn\;\;(J.P. Peng)}
\address{ Yuchao Tang$^*$:   School of Mathematics And Information Science, Guangzhou University,
Guangzhou, 510006, China}
\email{yctang@gzhu.edu.cn\;\;(Y.C. Tang)}

\thanks{$^*$ The corresponding author. This work was
supported by the National Natural Science Foundations of China (12031003, 12571558, 12571491), the Guangzhou Education Scientific Research Project 2024 (202315829) and the Jiangxi Provincial Natural Science Foundation (20224ACB211004), and the Postdoctoral Startup Foundation of Guangdong Province (62402153). }


\begin{abstract}
 This paper introduces a new definition of $\alpha$-monotone operators in real 2-uniformly convex and smooth Banach spaces. Based on this new definition, we establish several novel structural and analytical properties of such operators, which not only extend classical results from Hilbert spaces but also reveal new insights into the geometry of Banach spaces. In particular, we examine the resolvent of $\alpha$-maximal monotone operators and demonstrate how its behavior is consistent with, and generalizes, the well-known firmly nonexpansive property in the Hilbert space setting. Building upon this theoretical framework, we further investigate algorithmic applications. Specifically, we analyze the forward-reflected-backward splitting algorithm under the new $\alpha$-monotonicity assumption and prove its strong convergence as well as its $R$-linear convergence rate in real 2-uniformly convex and smooth Banach spaces.
\end{abstract}



\maketitle

\section{Introduction}

Let $X$ be a real Banach space and $X^{*}$ its dual space of $X$. Many problems arising in economics, mechanics, signal and image processing, and machine learning can be formulated as finding a point $x\in X$ satisfying the inclusion
\begin{equation}\label{single-monotone-inclusion}
\textrm{find } x\in X \quad \textrm{such that } 0\in Ax,
\end{equation}
where $A:X \rightarrow 2^{X^{*}}$ is a maximally monotone operator. A widely used iterative method for solving (\ref{single-monotone-inclusion}) is the proximal point algorithm (PPA), originally introduced by Martinet \cite{MB1970} and Rockafellar \cite{RT1976}. The PPA generates a sequence $x_n$ via
\begin{equation}
x_{n+1} = J_{\lambda A}x_n,
\end{equation}
where $\lambda >0$ and $J_{\lambda A}=(I+\lambda A)^{-1}$ denotes the resolvent operator of $A$. The PPA and its variants have been extensively studied in Hilbert spaces. See, e.g., \cite{GO1991,SS2000,AA2001,IL2022}  and the references therein. In particular, for a maximally monotone operator $A$, its resolvent operator $J_{\lambda A}$ is known to be firmly nonexpansive in Hilbert spaces. This property directly ensures the convergence of the PPA through the classical the Krasnoselskii-Mann iterative framework. In contrast to the case of Hilbert spaces, the resolvent of a maximally monotone operator is not, in general, a firmly nonexpansive mapping in Banach spaces. To address this issue, Kohsaka and Takahashi \cite{KT2008} introduced the class of firmly nonexpansive type mappings in Banach spaces. Specifically, they showed that if $A$ is a maximally monotone operator on a smooth, strictly convex, and reflexive Banach space $X$ such that $J:X\rightarrow X^{*}$ is the normalized duality mapping, $\lambda >0$, and $J_{\lambda A}= (J + \lambda A)^{-1}J$, then $J_{\lambda A}$ is a firmly nonexpansive type from a nonempty closed convex subset $C$ of $X$ onto the domain $D(A)$ of $A$. Further results and detailed discussions concerning monotone operators in Banach spaces are available in the existing literature \cite{BCA,SY,SYC,TP,TPS,TPCP}

Recently, Dao and Phan \cite{DP2019} introduced the concept of $\alpha$-monotone operators (See Definition \ref{classical def}) in Hilbert spaces. They studied various properties of $\alpha$-monotone operators with and without Lipschitz assumptions. Especially, the resolvent of $\alpha$-monotone operators is investigated in the Hilbert spaces setting. Let $A:X \rightarrow 2^{X^{*}}$, then Definition \ref{classical def} is also well-defined in Banach spaces. In particular, for $\alpha =0$, the $\alpha$-monotone operators reduce to the monotone operators, which was studied by \cite{KT2008} in Banach spaces.  For both $\alpha >0$ and $\alpha <0$, the approach employed in \cite{KT2008} fails to produce resolvent operator properties analogous to those associated with $\alpha$-monotone operators in Hilbert spaces of \cite{DP2019}. These limitations arise from a fundamental misalignment: the naive definition of $\alpha$-monotonicity is tied to the Hilbert space inner product, which simultaneously encodes the spaces norm and its duality. In Banach spaces, however, this duality is decoupled from the norm, and the normalized duality mapping $J:X\rightarrow X^{*}$ a central geometric object that generalizes the role of the inner product by linking elements of with their duals in $X^{*}$ must be explicitly incorporated in order to capture the intrinsic structure of the space. Without this incorporation, the notion of $\alpha$-monotonicity fails to reflect key geometric features such as convexity and smoothness. Consequently, the naive definition becomes inadequate for analyzing monotone operators and for developing or studying algorithms in the Banach space setting.

To address these gaps, we propose a novel definition of $\alpha$-monotonicity for operators in smooth Banach spaces (See Definition \ref{our def}) that explicitly integrates the normalized duality mapping $J$, aligning the operators monotonicity with the Banach spaces geometry. This definition resolves the shortcomings of the naive generalization: by replacing the Hilbert space norm squared \(\|x - y\|^2\) in Definition \ref{classical def} with the duality pairing \(\langle x - y, Jx - Jy \rangle\), it inherently encodes the Banach space's smoothness and convexity properties. As a result, in Hilbert spaces, $J$ reduces to the identity mapping, so Definition \ref{our def} recovers the standard $\alpha$-monotonicity-ensuring backward compatibility. For general smooth Banach spaces, the term \(\langle x - y, Jx - Jy \rangle\) acts as a ``Banach-space-adapted inner product proxy," enabling meaningful quantification of monotonicity that aligns with the spaces geometry. We summarize the main contributions of this paper as follows:

(i) We provide a systematic characterization of $\alpha$-monotone operators in smooth Banach spaces. In particular, we clarify their fundamental properties, establish their connections with existing operator classes such as maximally monotone and strongly monotone operators, and identify sufficient conditions under which $\alpha$-monotone operators are maximal.

(ii) We demonstrate that the resolvent associated with $\alpha$-monotone operators in real 2-uniformly convex and uniformly smooth Banach spaces preserves key contractive-type properties. These properties serve as natural analogues of the contractive behavior of resolvents in Hilbert spaces, thereby showing that our framework extends classical results in a consistent and meaningful way.

(iii) We further apply the concept of $\alpha$-monotonicity to analyze the forward-reflected-backward splitting algorithm in Banach spaces. Under the assumption that the ``strong convexity outweighs weak convexity" condition holds, we establish both strong convergence and $R$-linear convergence of the algorithm. This provides new theoretical guarantees for the efficiency and robustness of operator-splitting methods beyond the Hilbert space setting.

The remainder of the paper is organized as follows. Section 2 reviews preliminary concepts from Banach space geometry and monotone operator theory. Section 3 introduces our definition of $\alpha$-monotonicity, establishes its fundamental properties, and investigates the resolvent of $\alpha$-monotone operators, deriving corresponding contractivity results. Section 4 applies the proposed framework to the forward-reflected-backward splitting algorithm and proves the convergence of the associated algorithm in Banach spaces.

\section{Preliminaries}

Throughout this paper, every Banach space is assumed to be real. Let $\mathbb{N}$ and $\mathbb{R}$ denote the sets of positive integers and real numbers, respectively. For a sequence \({x_n}\) in a Banach space $X$, the strong convergence and the weak convergence of ${x_{n}}$ to $x \in X$ are denoted by \(x_{n} \to x\) and \(x_{n} \rightharpoonup x\), respectively.

A Banach space $X$ is said to be smooth if
\begin{align}\label{smooth}
  \lim _{t \to 0} \frac{\| x+t y\| -\| x\| }{t}
\end{align}
exists for all \(x, y \in S_X\), where \(S_X\) is the unit sphere of $X$. A Banach space $X$ is said to be strictly convex if
\begin{align}
  \left\|\frac{x+y}{2} \right\|<1
\end{align}
whenever \(x, y \in S_X\) and \(x \neq y\).

The modulus of convexity of \(X\), is
\[\delta_{X}(\varepsilon)=\inf \left\{1-\frac{\| x+y\| }{2}:\| x\| =\| y\| =1,\| x-y\| =\varepsilon\right\},\]
where the infimum can be taken over all $\|x\|\leq 1, \|y\|\leq 1$. The
modulus of smoothness is
\[\rho_{X}(\tau)=sup \left\{\frac{\| x+\tau y\| +\| x-\tau y\| }{2}-1: \| x\| =\| y\| =1\right\},\]
where the supremum can be taken over all $\|x\|\leq 1, \|y\|\leq 1$.

A Banach space X is called uniformly convex if \(\delta_{X}(\varepsilon)>0\) for all \(\varepsilon>0\), and is called q-uniformly convex \((2 \leq q<\infty)\) ) if there exists a constant \(C>0\) such that \(\delta_{X}(\varepsilon) \geq C \varepsilon^{q}\) for all \(\varepsilon>0\). X is called uniformly smooth if \(\rho_{X}(\tau) / \tau \to 0\) as \(\tau \to 0\), and is called p-uniformly smooth \((1<p \leq 2)\) if there exists a constant \(K>0\) such that \(\rho_{X}(\tau) \leq K \tau^{p}\) for all \(\tau>0\).


Let $X$ be a Banach space and let \(J: X \to X^{*}\) be the normalized duality mapping defined by
\begin{align}\label{dual map}
  J x=\left\{x^{*} \in X^{*}:\left< x, x^{*}\right>=\| x\| ^{2}=\left\| x^{*}\right\| ^{2}\right\}
\end{align}
for all \(x \in X\).

We know the following properties of $J$ (See for example Cioranescu \cite{CI}, and Takahashi \cite{TW} ) :

1. If $X$ is smooth, then $J$ is single-valued;

2. If $X$ is reflexive, then $J$ is onto;

3. If $X$ is strictly convex, then $J$ is one-to-one; that is, \(J x \cap J y \neq \emptyset\) implies that $x=y$;

4. If $X$ is strictly convex, then $J$ is strictly monotone, that is, if ( \((x, x^{*})\) , \((y, y^{*}) \in \mathrm{gra}J\) and \(\langle x-y, x^{*}-y^{*}\rangle=0\), then \(x=y\).

Let $X$ be a smooth Banach space. Following Alber \cite{AY} and Kamimura and Takahashi \cite{KT}, let \(\phi: X \times X \to \mathbb{R}\) be the mapping defined by
\begin{align}\label{phi}
  \phi(x, y)=\| x\|^{2} - 2\langle x, Jy\rangle + \| y\|^{2}
\end{align}
for all \(x, y \in X\). Note that $\phi$ is the Bregman distance corresponding to \(\|\cdot\|^{2}\); see Bregman \cite{BM}, Butnariu and Iusem \cite{BI}, and Censor and Lent \cite{CL}. If $X$ is a Hilbert space, then we have \(\phi(x, y)=\|x-y\|^{2}\) for all \(x, y \in X\). We know that
\begin{align}
  (\| x\| -\| y\| )^{2} \leq \phi(x, y) \leq(\| x\| +\| y\| )^{2}
\end{align}
for all \(x,y \in X\). If $X$ is strictly convex, then
\begin{align}
  \phi(x, y)=0 \Leftrightarrow x=y.
\end{align}






\begin{lemma}\label{phiequ}\cite{AY, AK, KT2}
  Let $X$ be a real smooth Banach space. Then the following identities hold:

 \emph{(i)} $\phi(x,y) = \phi(x,z) + \phi(z,y) + 2\langle x-z, Jz-Jy\rangle, \forall x,y,z\in X$.

  \emph{(ii)} $\phi(x,y) + \phi(y,x) = 2\langle x-y, Jx-Jy\rangle, \forall x,y\in X$.

  \emph{(iii)} $\langle x-y, Jz-Jw \rangle = \frac{1}{2}\{\phi(x,w) + \phi(y,z) - \phi(x,z) - \phi(y,w)\}, \forall x,y,z,w \in X$.
\end{lemma}

\begin{lemma}\label{2uc}\cite{BB,BK,THK}
   The space $X$ is 2-uniformly convex if and only if there exists \(\mu \geq 1\) such that
\begin{align}\label{mu}
  \frac{\| x+y\| ^{2}+\| x-y\| ^{2}}{2} \geq\| x\| ^{2}+\left\| \mu^{-1} y\right\| ^{2}
\end{align}
for all \(x, y \in X\).
 \end{lemma}
The minimum value of the set of all \(\mu \geq 1\) satisfying (\ref{mu}) for all \(x, y \in X\) is denoted by \(\mu_{X}\) and is called the 2-uniform convexity constant of $X$. It is obvious that \(\mu_{X}=1\) when $X$ is a Hilbert space.

\begin{lemma}\label{phiinequ1}\cite{AK}
  Suppose that Banach space $X$ is 2-uniformly convex and smooth. Then there exists $\mu>1$ such that
\[\frac{1}{\mu}\| x-y\|^{2} \leq \phi(x, y) \]
for all \(x, y \in X\).
\end{lemma}


\section{$\alpha$-monotone operators in Banach spaces}

In this section, we introduce a new definition of \(\alpha\)-monotone operators in Banach spaces and compare it with the classical version. And then, we derive some basic properties of these new \(\alpha\)-monotone operators.

\subsection{New definition of monotone operators in Banach spaces}
In this subsection, we introduce a new definition of \(\alpha\)-monotone operators in Banach spaces and compare it with the classical version. Specifically, an operator that is monotone in the sense of our definition is not equivalent to one that is monotone in the sense of the classical definition. We further show that in a 2-uniformly convex and smooth Banach space, an \(\alpha\)-monotone operator under our definition is equivalent to that under the classical version if and only if the space is 2-uniformly smooth. We also establish that strong monotonicity of an operator in our definition implies strong monotonicity in the classical definition, though the converse does not hold; similarly, weak monotonicity in the classical definition implies weak monotonicity in our definition, with the converse also failing.

Now we recall the classical version of $\alpha$-monotonicity operators in Banach spaces.

\begin{definition}[Classical $\alpha$-monotonicity operators in Banach spaces]\label{classical def}
   An operator \(A: X \rightrightarrows X^*\) is said to be $\alpha$-monotone \((\alpha\in \mathbb{R})\) if
\[ \forall(x, u),(y, v) \in~\mathrm{gra}~A, \quad\langle x-y, u-v\rangle \geq \alpha\| x-y\|^2.\]
\end{definition}

If Banach space $X$ is smooth, then the dual mapping $J: X\to X^*$ is single-valued. We can define a new definition of $\alpha$-monotonicity operators in smooth Banach spaces as follows.
\begin{definition}[$\alpha$-monotonicity operators in smooth Banach
spaces]\label{our def}
   Let $X$ be a smooth Banach space. An operator \(A: X \rightrightarrows X^*\) is said to be $\alpha$-monotone \((\alpha \in \mathbb{R})\) if
\[ \forall(x, u),(y, v) \in~\mathrm{gra}~A, \quad\langle x-y, u-v\rangle \geq \alpha\langle x-y, Jx-Jy\rangle.\]
 \end{definition}
The constant $\alpha$ is referred to as the monotonicity constant.
We say that $A$ is monotone if \(\alpha=0\), strongly monotone if \(\alpha>0\) , and weakly monotone if \(\alpha<0\). The operator $A$ is said to be maximally $\alpha$-monotone if it is $\alpha$-monotone and there is no $\alpha$-monotone operator \(B: X \rightrightarrows X^*\) such that $\mathrm{gra}~B$ properly contains $\mathrm{gra}~A$.

\begin{remark}\label{equi case}
  It is straightforward to show that in Banach spaces, an operator being monotone in the sense of Definition \ref{our def} ($\alpha =0$) is equivalent to its being monotone in the sense of Definition \ref{classical def}. Hence, for a monotone operator, the property of being maximal is also equivalent under both Definition \ref{our def} and Definition \ref{classical def}.
\end{remark}

\begin{example}
  Let $X$ be a smooth Banach space. The dual mapping $J: X\to X^*$ is 1-monotone under Definition \ref{our def}, hence $J$ is strong monotone under Definition \ref{our def}. We know that if $f(x) = \frac{1}{2}\|x\|^2$, then $\nabla f = J$, hence $\nabla f$ is strong monotone under Definition \ref{our def}.
\end{example}

\begin{theorem}\label{main inequa1}
  Let $X$ be a real smooth and $2$-uniformly convex Banach space. Then, there exists \(\mu \geq 1\) such that
\[\frac{1}{2\mu}\| x-y\| ^{2} \leq \langle x-y, Jx-Jy \rangle,~~ \forall x, y \in X .\]
\end{theorem}
\begin{proof}
  By Lemma \ref{phiequ}, $\langle x-y, Jx-Jy \rangle = \frac{1}{2} \phi(x,y) + \frac{1}{2} \phi(y,x) \geq \frac{1}{2} \phi(x,y)$. Then, by Lemma \ref{phiinequ1}, there exists \(\mu \geq 1\) such that $\frac{1}{\mu}\| x-y\|^{2} \leq \phi(x, y)$.  Hence $\frac{1}{2\mu}\| x-y\| ^{2} \leq \langle x-y, Jx-Jy \rangle$.
\end{proof}

Now we want to show that if $X$ be a real smooth and $2$-uniformly convex Banach space, there exist \(\mu, L \geq 1\) such that $\frac{1}{2\mu}\| x-y\|^{2} \leq \langle x-y, Jx-Jy \rangle \leq L\| x-y\| ^{2}$, i.e., for a monotone operator, the property of being strong or weak monotone is also equivalent under both Definition \ref{our def} and Definition \ref{classical def}, then $X$ is isomorphic to a Hilbert space.
\begin{theorem}\label{main 1}
   Suppose that Banach space $X$ is 2-uniformly convex and smooth, and there exists $L>0$ such that
\[\langle x-y, Jx-Jy \rangle \leq L\|x-y\|^2\]
for all \(x, y \in X\). Then $X$ is isomorphic to a Hilbert space.
\end{theorem}

\begin{remark}
Let $X$ be a smooth and $2$-uniformly convex Banach space.
By combining Theorem \ref{main inequa1} and Theorem \ref{main 1}, we proceed to compare the two versions of \(\alpha\)-monotone operators in Banach spaces.

(i) Strong monotonicity of an operator under Definition \ref{our def} $\Rightarrow$ strong monotonicity under Definition \ref{classical def}, though the converse is false (for example $L_{p}[0,1]~(1<p<2)$ is smooth and 2-uniformly convex Banach space but not isomorphic to a Hilbert space.).

(ii) Weak monotonicity under Definition \ref{classical def} $\Rightarrow$ weak monotonicity under Definition \ref{our def}, yet the converse fails.

(iii) It should be pointed out that strong monotone operators under Definition \ref{our def} is ``dense'' in strong monotone operators under Definition \ref{classical def}, see the following Theorem \ref{dense2}.
\end{remark}

We claim that if we have the following theorem, we can show Theorem \ref{main 1} holds.
\begin{theorem}\label{main 2}
  Suppose that Banach space $X$ is 2-uniformly convex and smooth, and there exists $K>0$ such that
\[K\| x-y\|^{2} \geq \phi(x, y) ~ \mathrm{or}~K\| x-y\|^{2} \geq \phi(y, x)\]
for all \(x, y \in X\). Then $X$ is isomorphic to a Hilbert space.
\end{theorem}

Now we prove Theorem \ref{main 1} by using Theorem \ref{main 2}.

\begin{proof}[proof of theorem \ref{main 1}]
  Since $\phi(x,y) +  \phi(y,x) = 2\langle x-y, Jx-Jy \rangle$, if $\langle x-y, Jx-Jy \rangle \leq L\|x-y\|^2$, then we have $\phi(x,y) +  \phi(y,x) \leq 2L\|x-y\|^2$. Therefore, $\phi(x,y) \leq L\|x-y\|^2$ or $\phi(y,x) \leq L\|x-y\|^2$. By Theorem \ref{main 2}, we get that $X$ is isomorphic to a Hilbert space.
\end{proof}

Before proving Theorem \ref{main 2}, we first establish some lemmas.
\begin{lemma}\label{2us}\cite{THK}
   The space $X$ 2-uniformly smooth if and only if there exists \(K \geq 1\) such that
\begin{align}\label{2-us}
  \frac{\| x+y\| ^{2}+\| x-y\| ^{2}}{2} \leq \| x\|^{2}+\left\| K y\right\|^{2}
\end{align}
for all \(x, y \in X\).
 \end{lemma}

\begin{remark}
  Let $u=x-y, v=x+y$, then (\ref{2-us}) is equivalent to
  \begin{align}
    \left\| \frac{u+v}{2}\right\|^{2} \geq \frac{\| u\| ^{2}+\| v\| ^{2}}{2}-\frac{K^{2}}{4}\| u-v\| ^{2}.
  \end{align}
\end{remark}

Next, we establish that in 2-uniformly smooth Banach spaces, the inequality in Theorem \ref{main 2} holds.
\begin{lemma}
  Suppose that Banach space $X$ is 2-uniformly smooth. Then there exists $L\geq1$ such that
\[L\| x-y\|^{2} \geq \phi(x, y) \]
for all \(x, y \in X\).
\end{lemma}

\begin{proof}
  If $X$ is 2-uniformly smooth, then By Lemma \ref{2us},  there exists \(K \geq 1\) such that
\begin{align}\label{st2}
  \left\| \frac{u+v}{2}\right\|^{2} \geq \frac{\| u\| ^{2}+\| v\| ^{2}}{2}-\frac{K^{2}}{4}\| u-v\|^{2}
\end{align}
for all \(u, v \in X\) . Let \(x, y \in X\). By (\ref{st2}) and induction, we can easily show that
\begin{equation}
\begin{aligned}
& \left\| \left(1-\frac{1}{2^{n}}\right) y+\frac{1}{2^{n}} x\right\| ^{2} \\
& \geq \left(1-\frac{1}{2^{n}}\right)\| y\| ^{2}+\frac{1}{2^{n}}\| x\| ^{2}-K^{2}\left(1-\frac{1}{2^{n}}\right) \frac{1}{2^{n}}\| y-x\| ^{2}
\end{aligned}
\end{equation}
for all \(n \in \mathbb{N}\) . Hence we have
\begin{equation}\label{stnn}
  \begin{aligned} & 2^{n}\left(\left\| y+\frac{1}{2^{n}}(x-y)\right\| ^{2}-\| y\| ^{2}\right) \\ & \geq-\| y\| ^{2}+\| x\| ^{2}-K^{2}\left(1-\frac{1}{2^{n}}\right)\| y-x\| ^{2} \end{aligned}
\end{equation}
for all \(n \in \mathbb{N}\) . The smoothness of $X$ implies that
\begin{align}\label{stss}
  2\langle x-y, J y \rangle=\lim_{t \to 0} \frac{\| y+t(x-y)\| ^{2}-\| y\| ^{2}}{t} .
\end{align}
By (\ref{stnn}) and (\ref{stss}), we have
\[\begin{aligned} 2\langle x-y, J y \rangle & =\lim _{n \to \infty} 2^{n}\left(\left\| y+\frac{1}{2^{n}}(x-y)\right\| ^{2}-\| y\| ^{2}\right) \\ & \geq-\| y\| ^{2}+\| x\| ^{2}-K^{2}\| x-y\| ^{2} . \end{aligned}\]

Therefore, we obtain \( L\|x-y\|^{2} \geq \phi(x, y)\) as desired, where $L=K^2\geq1$.
\end{proof}

Next, we want to show that, if $X$ is smooth Banach space, and if there exists $L>0$ such that $L\| x-y\|^{2} \geq \phi(x, y) $ for all \(x, y \in X\), then $X$ is 2-uniformly smooth.

\begin{lemma}\label{nonde}\cite{HRB}
  Let $g: X\to \mathbb{R}$ be a convex function. For any $x_0, x\in X$, the function \begin{align}
    t \mapsto \frac{g(x_0 + tx) - g(x_0)}{t}
  \end{align}
  is nondecreasing for $t>0$.
\end{lemma}

\begin{lemma}\label{key lemma 1}
  Suppose that Banach space $X$ is smooth, and there exists $K>0$ such that
\[K\| x-y\|^{2} \geq \phi(x, y)~ \mathrm{or}~K\| x-y\|^{2} \geq \phi(y, x)\]
for all \(x, y \in X\). Then $X$ is 2-uniformly smooth.
\end{lemma}

\begin{proof}

Let \(x, y \in X\). The smoothness of $X$ implies that
\begin{align}\label{stsss}
  2\langle x-y, Jy\rangle=\lim_{t \to 0} \frac{\| y+t(x-y)\| ^{2}-\| y\| ^{2}}{t} .
\end{align}
and if
\begin{align}\label{sinequ}
K\| x-y\|^{2} \geq \phi(x, y),
\end{align}

By (\ref{stsss}) and (\ref{sinequ}), and noting that $\phi(x, y) = \|x\|^2 - 2\langle x-y, Jy\rangle - \|y\|^2$, we have
\begin{equation}
\begin{aligned} \lim _{n \to \infty} 2^{n}\left(\left\| y+\frac{1}{2^{n}}(x-y)\right\| ^{2}-\| y\| ^{2}\right) & = 2\langle x-y, J y\rangle \\ & \geq-\| y\| ^{2}+\| x\| ^{2}-K\| x-y\| ^{2} .
\end{aligned}
\end{equation}
Then by Lemma \ref{nonde} we have
\begin{align}
  2\left(\left\|\frac{x+y}{2} \right\|^2 - \| y\|^{2}\right) \geq-\| y\| ^{2}+\| x\| ^{2}-K\| x-y\| ^{2} .
\end{align}
Hence we have
\begin{align}\label{st3}
  \left\| \frac{x+y}{2}\right\|^{2} \geq \frac{\| x\| ^{2}+\| y\| ^{2}}{2}-\frac{K}{2}\| x-y\|^{2} \geq \frac{\| x\| ^{2}+\| y\| ^{2}}{2}-\frac{K^2}{4}\| x-y\|^{2},
\end{align}
the last inequality of (\ref{st3}) holds since we can always suppose $K\geq 2$.
Therefore, by Lemma \ref{2us}, we conclude that $X$ is 2-uniformly smooth.

The proof for the case $K\| x-y\|^{2} \geq \phi(y, x)$ follows similar arguments.
\end{proof}

Finally, by combining Lemma \ref{key lemma 1} with the following two theorems in Banach spaces theory, we can establish Theorem \ref{main 2}.

\begin{theorem}\label{usuc type}\cite{BL}
  Let $X$ be a Banach space,

  (i) if $X$ is $p$-uniformly smooth, then $X$ has type $p$.

  (ii) if $X$ is $q$-uniformly convex, then $X$ has cotype $q$.
\end{theorem}

\begin{theorem}\label{type H}\cite{AKNJ}
  A Banach space $X$ is of type $2$ and cotype $2$ if and only if $X$ is isomorphic to a Hilbert space.
\end{theorem}

Now we prove Theorem \ref{main 2}:
\begin{proof}[Proof of theorem \ref{main 2}]
  By Lemma \ref{key lemma 1}, if a Banach space $X$ is smooth, and there exists a constant $K > 0$ such that for all $x, y \in X$, either
\begin{equation}
K \| x - y \|^2 \geq \phi(x, y) \quad \text{or} \quad K \| x - y \|^2 \geq \phi(y, x),
\end{equation}
then \( X \) is $2$-uniformly smooth. Moreover, as \( X \) is $2$-uniformly convex, it follows from Theorem \ref{usuc type} that \( X \) has type $2$ and cotype $2$. Finally, invoking Theorem \ref{type H}, we conclude that $X$ is isomorphic to a Hilbert space.
\end{proof}

\subsection{Basic properties of new monotone operators in Banach spaces}

In this subsection, we follow  Dao and Phan's idea \cite{DP2019} to give some basic properties of $\alpha$-monotone operators.

Throughout, $\alpha$-monotonicity of an operator is defined as in Definition \ref{our def}, unless otherwise specified for Definition \ref{classical def}.

Let \(A: X \rightrightarrows X\) be an operator on Banach space $X$ . Then its domain is $\mathrm{dom}~A=\{x \in X~|~ Ax \neq \emptyset\}$ , its set of zeros is $\mathrm{zer}~A=\{x \in X~ |~ 0 \in Ax\}$. The graph of $A$ is the set $\mathrm{gra}A:=\{(x, u) \in X\times X^*~|~ u \in Ax\}$ and the inverse of $A$ , denoted by \(A^{-1}\), is the operator with graph $\mathrm{gra}A^{-1}:=\{(u, x) \in X\times X^*~|~u \in Ax \}$ . The resolvent of $A$ is defined by
\[J_{A}:=(J+A)^{-1}J,\]
where $J$ is the dual mapping.

\begin{lemma}[monotonicity versus $\alpha$-monotonicity]\label{mono}
  Let \(A: X \rightrightarrows X^*\) and let $\alpha, \beta \in \mathbb{R}$. Then the following hold:

  (i) $A$ is $\alpha$-monotone if and only if \(A-\beta J\) is \((\alpha-\beta)\)-monotone.

  (ii) $A$ is maximally $\alpha$-monotone if and only if \(A-\beta J\) is maximally \((\alpha-\beta)\)-monotone.

  Consequently, $A$ is (resp., maximally) $\alpha$-monotone if and only if \(A-\alpha J\) is (resp., maximally) monotone.
\end{lemma}
\begin{proof}
  (i) We first have the equivalences
\[ \quad(x, u) \in~\mathrm{gra}~A \Leftrightarrow(x, u-\beta Jx) \in gra(A-\beta J),\]

\[ \quad(y, v) \in~\mathrm{gra}~A \Leftrightarrow(y, v-\beta Jy) \in gra(A-\beta J),\]
and
\begin{align*}
  &\langle x-y, u-v\rangle ~\geq~ \alpha\langle x-y, Jx-Jy\rangle.\\
 \Leftrightarrow &\langle x-y,(u-\beta Jx)-(v-\beta Jy)\rangle ~\geq~(\alpha-\beta)\langle x-y, Jx-Jy\rangle.
\end{align*}
from which the conclusion follows.

(ii) Assume that $A$ is maximally $\alpha$-monotone. By (i),~$A-\beta J$ is \((\alpha-\beta)\)-monotone. Now, suppose that \(A-\beta J\) is not maximally \((\alpha-\beta)\)-monotone. Then there must exist \(B': X \rightrightarrows X^*\) such that \(B'\) is \((\alpha-\beta)\)-monotone and \(\mathrm{gra}~(A-\beta J) \subsetneq \mathrm{gra}~B'\). It follows that \(B:=B'+\beta J\) is $\alpha$-monotone due to (i) and that \(\mathrm{gra}A~\subsetneq \mathrm{gra}~B\) , which contradict the maximal $\alpha$-monotonicity of $A$. We deduce that if $A$ is maximally $\alpha$-monotone, then \(A-\beta J\) is maximally \((\alpha-\beta)\)-monotone. This also implies that if \(A-\beta J\) is maximally \((\alpha-\beta)\)-monotone, then \(A=(A-\beta J)+\beta J\) is maximally $\alpha$-monotone, and we are done.
\end{proof}

\begin{definition}
  We say that $A: X\rightrightarrows X$ is $\sigma$-firmly nonexpansive-type mapping if \(\sigma>0\) and
  \begin{align}
  \forall~(x, u),~(y, v) \in \mathrm{gra}~A, \quad\langle Jx-Jy, u-v\rangle~\geq~\sigma \langle u-v, Ju-Jv\rangle.
  \end{align}
  If $\sigma = 1$, it reduces to the firmly nonexpansive-type mapping, which was introduced by  Kohsaka and Takahashi \cite{KT2}.
\end{definition}

\begin{lemma}[resolvents of $\alpha$-monotone operators]\label{reso}
   Let \(A: X \rightrightarrows X^*\) be $\alpha$-monotone and let \(\gamma >0\). Then the following hold:

   (1) For all \((x, a)\) , \((y, b) \in \mathrm{gra}~J_{\gamma A}\) ,
\[ \langle Jx-Jy, a-b \rangle~\geq~(1+\gamma \alpha)\langle a-b, Ja-Jb\rangle.\]

(2) If \(J_{\gamma A}\) is single valued, then, for all \(x, y \in dom J_{\gamma A}\),
\[ \quad\langle Jx-Jy, J_{\gamma A} x-J_{\gamma A} y\rangle~ \geq(1+\gamma \alpha) \langle J_{\gamma A} x-J_{\gamma A} y, JJ_{\gamma A} x - JJ_{\gamma A} y\rangle.\]
i.e., \(J_{\gamma A}\) is \((1+\gamma \alpha)\)-firmly nonexpansive-type.
\end{lemma}

\begin{proof}
  (i) Let \((x, a), (y, b) \in \mathrm{gra}~J_{\gamma A}\). Then \(Jx \in(J+\gamma A) a\), \(Jy \in(J+\gamma A) b\), and so \(Jx=Ja+\gamma u\), \(Jy=Jb+\gamma v\) for some \(u \in A a\), \(v \in A b\). We derive from the $\alpha$-monotonicity of $A$ that
  \begin{align*}
  \langle Jx-Jy, a-b\rangle&=\langle(Ja+\gamma u)-(Jb+\gamma v), a-b\rangle\\
  &=\langle a-b, Ja-Jb\rangle~+~\gamma\langle a-b, u-v\rangle\\
  &\geq \langle a-b, Ja-Jb\rangle~+~\gamma \alpha\langle a-b, Ja-Jb\rangle \\
  &=(1+\gamma \alpha)\langle a-b, Ja-Jb\rangle.
  \end{align*}

 (ii) This is a direct consequence of (i).
\end{proof}

\begin{remark}
  (i) Clearly, if $A$ is $\sigma$-firmly nonexpansive-type mapping, then $A$ is single-valued.  By the positive homogeneity of $J$, if $A$ is $\sigma$-firmly nonexpansive-type mapping, then $\sigma A$ is firmly nonexpansive-type mapping.

  (ii) Let \(A: X \rightrightarrows X^*\) be $\alpha$-monotone and let \(\gamma >0 \). Then $J_{\gamma A}$ is $(1+\gamma \alpha)$-firmly nonexpansive-type mapping, and by the positive homogeneity of $J$, $(1+\gamma \alpha) J_{\gamma A}$ is firmly nonexpansive-type mapping. Particularly, if $A$ is monotone, i.e., $\alpha =0$, then $J_{\gamma A}$ is firmly nonexpansive-type mapping.
\end{remark}

\begin{lemma}\label{full domain}(Browder \cite{BF} and Rockafellar \cite{RRT}; see also Barbu \cite{BV} and Takahashi \cite{TW0})
   Let $X$ be a smooth, strictly convex and reflexive Banach space and let \(A : X\to 2^{X^*}\) be a monotone operator. $A$ is maximal monotone under definition \ref{classical def} if and only if \(R(J+r A)=X^{*}\) for all \(r>0\).
\end{lemma}

\begin{remark}
  If $X$ be a smooth, strictly convex and reflexive Banach space, $J$ is single-valued, one-to-one and onto, hence $A$ is maximal monotone under Definition \ref{classical def} if and only if $dom J_{\gamma A} = X$. And since $A$ is maximal monotone under Definition \ref{classical def} if and only if $A$ is maximal monotone under Definition \ref{our def}. Hence, $A$ is maximal monotone under Definition \ref{our def} if and only if $dom J_{\gamma A} = X$.
\end{remark}

\begin{proposition}[single-valuedness and full domain]\label{cvld}
   Let $X$ be a smooth, strictly convex and reflexive Banach space. Let \(A: X \rightrightarrows X^*\) be $\alpha$-monotone and let \(\gamma >0\) such that \(1+\gamma \alpha>0\) . Then the following hold:

  (i) \(J_{\gamma A}\) is single valued.

  (ii) $\mathrm{dom}~J_{\gamma A}=X$ if and only if $A$ is maximally $\alpha$-monotone.
\end{proposition}
\begin{proof}
  (i) If $X$ is strict convex, then $J$ is strictly monotone. Hence (i) follows from Lemma \ref{reso} (i).

  (ii) By Lemma \ref{mono} (i), \(A':=A-\alpha J\) is monotone. Noting that \((\beta T)^{-1}=T^{-1} \circ \frac{1}{\beta}Id\)  for any operator $T$ and any \(\beta \in \mathbb{R}\setminus\{0\}\) , we have
  \begin{align*}
  J_{\gamma A}=(J+\gamma A)^{-1}\circ J &= ((1+\gamma \alpha)J + \gamma(A-\alpha J))^{-1}\circ J\\
  &=\left( J + \frac{\gamma}{1+\gamma \alpha} A'\right)^{-1} \circ\left(\frac{1}{1+\gamma \alpha} J\right)\\
  &=\left(\frac{1}{1+\gamma \alpha} \right)J_{\frac{\gamma}{1+\gamma \alpha} A'}.
  \end{align*}

  It follows that
  \begin{align*}
    \mathrm{dom}~J_{\gamma A}=X &\Leftrightarrow \mathrm{dom}~J_{\frac{\gamma}{1+\gamma \alpha} A'}=X\\
    & \Leftrightarrow A' \mathrm{~is~maximally~ monotone~(by~Lemma~\ref{full domain})}\\
    & \Leftrightarrow A \mathrm{~is~ maximally}~ \alpha\mathrm{-monotone~ (by~ Lemma~\ref{mono})}.
  \end{align*}

\end{proof}

\subsection{Strong monotone operators under Definition \ref{our def} are dense in strong monotone operators under Definition \ref{classical def}}

In this subsection, we show that the maximal strong monotone operators under new definition is dense in the maximal strong monotone operators under classical version under the following sense (see Theorem \ref{dense2}). For sake of convenience, let \textbf{csm} be denoted by the set of maximal strong monotone operators under Definition \ref{classical def}, and let \textbf{nsm} be denoted by the set of maximal strong monotone operators under Definition \ref{our def}, respectively. We always assume that the Banach space $X$ is strict convex, smooth and reflexive in this subsection.

\begin{proposition}\label{single point}
 Let $A \in \textbf{csm}$ or $A \in \textbf{nsm}$, if $A^{-1}0 \neq \emptyset$, then $A^{-1}0$ is single point.
\end{proposition}
\begin{proof}
    Suppose that $A^{-1}0$ contains more than one point. Let $x_1 \neq x_2 \in D(A)$, such that $x_1, x_2 \in A^{-1}0$. Then, if $A \in \textbf{csm}$ or $A \in \textbf{nsm}$, we have $\langle 0-0, x_1 - x_2 \rangle \geq \alpha \|x_1 - x_2\|^2 > 0$ or $\langle 0-0, x_1 - x_2 \rangle \geq \alpha \langle x_1 - x_2£¬ Jx_1 - Jx_2 \rangle >0$, which is a contradiction. Hence, $A^{-1}0$ must be a single point.
\end{proof}

\begin{lemma}\label{exists}
    Let $A \in \textbf{csm}$ and $\alpha >0$,  then $A+\alpha J \in \textbf{nsm}$, and $(A+\alpha J)^{-1}0$ exists and is a single point.
\end{lemma}

\begin{proof}
    Let $A \in \textbf{csm}$ and let $\alpha >0$. By Definition \ref{our def}, we know that $A+\alpha J \in \textbf{nsm}$. Therefore, if $(A+\alpha J)^{-1}0 \neq \emptyset$, then $(A+\alpha J)^{-1}0$ is a single point by Proposition \ref{single point}. Next, we show that $(A+\alpha J)^{-1}0 \neq \emptyset$. Since $A+\alpha J = \alpha (\frac{1}{\alpha}A + J)$, and by Lemma \ref{full domain}, we have $R(\frac{1}{\alpha}A + J) = X^*$. Hence $R(A+\alpha J) = X^*$, and consequently $(A+\alpha J)^{-1}0$ exists.
\end{proof}

\begin{theorem}\label{dense2}
Let $A \in \textbf{csm}$ and $A^{-1}0 \neq \emptyset$, in this case, $A^{-1}0$ is single point, we denote $\{x^*\} = A^{-1}0$. Then there exist $\{A_n\}_{n}^{\infty} \subset \textbf{nsm}$ such that $A_{n}^{-1}0 \neq \emptyset$ and it is single point. Furthermore, let $\{x_n\} = A_{n}^{-1}0$, then we have $\|x_n - x^*\| \to 0$ as $n\to \infty$.
\end{theorem}

\begin{proof}
     Let $A_n = A+ \alpha_n J$, where $\alpha_n >0$ and $\alpha_n \to 0$ as $n\to \infty$. Then $A_n$ is $\alpha_n$-monotone under Definition \ref{our def}.

     By Lemma \ref{exists}, if $A\in \textbf{csm}$, then there exists $x_n$ such that $0 \in (A+\alpha_n J)x_n$. Next, we show that $x_n$ is bounded. In fact, there exists $y_n \in Ax_n$ such that $y_n = -\alpha_n Jx_n$. Let $x^* \in A^{-1}0$. Then $\langle x_n - x^*, y_n - 0\rangle \geq \alpha\|x_n - x^*\|^2 \geq 0$, whici implies $\langle x_n - x^*, -\alpha_n Jx_n\rangle \geq 0$. Hence  $\langle x^*, \alpha_n Jx_n\rangle \geq \langle x_n, \alpha_n Jx_n\rangle = \alpha_n\| x_n\|^2$, and thus $\|x_n\| \leq \|x^*\|$. Therefore, the sequence $\{x_n\}$ is bounded.

 On the other hand,  from $\langle x_n - x^*, y_n - 0\rangle \geq \alpha\|x_n - x^*\|^2 $, we obtain $\|x_n - x^*\| \leq \frac{\alpha_n}{\alpha}\|x_n\| \leq \frac{\alpha_n}{\alpha}\|x^*\| \to 0$ as $n\rightarrow \infty$.
\end{proof}




\subsection{The class of nonlinear mappings related to maximal monotone operators in Banach spaces}

In this subsection, we study some properties of $\sigma$-firmly nonexpansive type mappings. We first show that the class of $\sigma$-firmly nonexpansive type mappings coincides with that of resolvents of $\alpha$-monotone operators in Banach spaces. $\alpha$-monotonicity of an operator is defined as in Definition \ref{our def}, unless otherwise specified for Definition \ref{classical def}.

\begin{proposition}\label{f and J}
  Let $X$ be a smooth, strictly convex and reflexive Banach space, let $C$ be a nonempty closed convex subset of $X$ and let $T$ be a mapping from $C$ into itself. Then the following are equivalent:

  \emph{(i)} T is of $(1+\alpha)$-firmly nonexpansive type;

  \emph{(ii)} there exists a $\alpha$-monotone operator \(A \subset X\times X^{*}\) such that
  \begin{align}
    D(A) \subset C \subset J^{-1} R(J+A)
  \end{align}
  and $Tx=(J+A)^{-1} Jx$ for all $x \in C$.
\end{proposition}
\begin{proof}
   Suppose that there exists a $\alpha$-monotone operator $A \subset X\times X^{*}$ such that $D(A) \subset C \subset J^{-1} R(J+A)$ and $T x=(J+A)^{-1} J x$ for all $x \in C$. Since $X$ is smooth and strictly convex and reflexive, by Proposition \ref{cvld}, $T$ is single-valued. By $D(A) \subset C \subset J^{-1} R(J+A)$, $T$ is a mapping from $C$ into itself. Let $x,y \in C$. Then we have
   $Jx-JTx \in ATx$ and $Jy-JTy \in ATy$. Since $A$ is $\alpha$-monotone, we have
   \begin{align}
     \langle Tx - Ty, Jx-JTx - (Jy-JTy)\rangle \geq \alpha \langle Tx - Ty, JTx - JTy\rangle.
   \end{align}
   Hence, we get $\langle Tx - Ty, Jx-Jy\rangle \geq (1+\alpha) \langle Tx - Ty, JTx - JTy\rangle$. This shows that $T$ is of $(1+\alpha)$-firmly nonexpansive type.

   We next show that (i) implies (ii). Suppose that $T$ is of $(1+\alpha)$-firmly nonexpansive type. Let $A \subset X\times X^{*}$ be the set-valued mapping defined by $A=JT^{-1}-J$, where $T^{-1}$ is defined by
   \begin{equation}
   \begin{aligned}
     T^{-1} u= \begin{cases}\{v \in C: T v=u\}, & u \in R(T) ; \\ \emptyset,  & \textrm{otherwise}.\end{cases}
   \end{aligned}
   \end{equation}
   It is obvious that $Tx = (J+A)^{-1} Jx$ for all $x \in C$. We show that $A$ is $\alpha$-monotone. Let $(x_{i}, x_{i}^{*}) \in A$ be given $(i=1,2)$. Then we have
   \begin{equation}
     \begin{aligned}
     x_{i}^{*} \in A x_{i} & \Leftrightarrow x_{i}^{*} \in JT^{-1} x_{i}-J x_{i} \\ & \Leftrightarrow x_{i}^{*}+J x_{i} \in JT^{-1} x_{i} \\
     & \Leftrightarrow J^{-1}\left(x_{i}^{*}+Jx_{i}\right) \in T^{-1} x_{i} \\
     & \Leftrightarrow TJ^{-1}\left(x_{i}^{*}+Jx_{i}\right)=x_{i}
     \end{aligned}
   \end{equation}
   for \(i=1,2\) . Putting $u_{i}=J^{-1}(x_{i}^{*}+J x_{i})$, we have
   $Tu_{i} = x_{i}$ and $Ju_{i} - Jx_{i} = x_{i}^{*}~~(i=1,2)$. Since $T$ is of $(1+\alpha)$-firmly nonexpansive type, we have
   \begin{equation}
     \begin{aligned}
     \left< x_{1}-x_{2}, x_{1}^{*}-x_{2}^{*}\right>
     & =\left< Tu_{1} - Tu_{2}, x_{1}^{*}-x_{2}^{*}\right> \\
     & =\left< Tu_{1} - Tu_{2}, Ju_{1} - Jx_{1} - \left(Ju_{2} - J x_{2}\right)\right> \\
     & =\left< Tu_{1} - Tu_{2}, Ju_{1} - JTu_{1} - \left(Ju_{2} - JT u_{2}\right)\right> \\
     &\geq (1+\alpha-\alpha) \left< Tu_{1} - Tu_{2}, JTu_{1} - JTu_{2}\right>\\
     &\geq \alpha \left< x_{1} - x_{2}, Jx_{1} - Jx_{2}\right>.
     \end{aligned}
   \end{equation}
   Thus $A$ is $\alpha$-monotone. We finally show that $D(A) \subset C \subset J^{-1} R(J+A)$. It is easy to see that $D(A) = D(JT^{-1} - J) = D(T^{-1}) = R(T) \subset C$. Since $J+A=JT^{-1}$, we also have
   \begin{align}
     R(J+A) = R\left(JT^{-1}\right) = D\left(\left(J T^{-1}\right)^{-1}\right) = D\left(TJ^{-1}\right) = JD(T) = JC.
   \end{align}
   Thus $C=J^{-1} R(J+A)$. This completes the proof.
\end{proof}

\begin{lemma}\label{l-sigma-f}
   Let $X$ be a smooth Banach space, let $C$ be a nonempty closed convex subset of $X$, and let $T: C \to C$ be a mapping. Then $T$ is of $\sigma$-firmly nonexpansive type if and only if
   \begin{align}\label{sigma-f}
     \sigma\left( \phi(T x, T y)+\phi(T y, T x)\right)+\phi(T x, x)+\phi(T y, y) \leq \phi(T x, y)+\phi(T y, x)
   \end{align}
for all $x,y \in C$.
\end{lemma}
\begin{proof}
  Let $x, y \in C$. Since $T$ is of $\sigma$-firmly nonexpansive type,
  \begin{align}\label{Ts}
    \sigma \langle Tx - Ty, JTx - JTy\rangle \leq \langle Tx - Ty, Jx - Jy\rangle.
  \end{align}
 It follows from Lemma ($\ref{phiequ}$) $(iii)$ that (\ref{Ts}) is equivalent to
  \begin{equation}
    \begin{aligned} & \sigma \frac{1}{2}\{\phi(T x, T y)+\phi(T y, T x)-\phi(T x, T x)-\phi(T y, T y)\} \\ & \leq \frac{1}{2}\{\phi(T x, y)+\phi(T y, x)-\phi(T x, x)-\phi(T y, y)\} .
    \end{aligned}
  \end{equation}
This is also equivalent to (\ref{sigma-f}). This completes the proof.
\end{proof}

\begin{theorem}\label{sc of T}
  Let $X$ be strict convex and smooth and reflexive Banach space, let $C$ be a nonempty closed convex subset of $X$, and let $\sigma >1$, $T: C \to C$ be a $\sigma$-firmly nonexpansive-type mapping such that $F(T)$ is nonempty. Then $F(T)$ is single point and $\{T^{n}x\}$ converge strongly to fixed point of $T$ for any $x\in C$. If $X$ is $2$-uniformly convex, then there is a constant $\mu$ such that $\|u - T^{n}x\|^2 \leq \frac{\mu}{\sigma^n} \phi(u,x)$.
\end{theorem}
\begin{proof}
   Let $(x, u) \in C\times F(T)$ be given. Then it follows from Lemma \ref{l-sigma-f} that
   \begin{align}
     \sigma\left( \phi(T x, T y)+\phi(T y, T x)\right)+\phi(T x, x)+\phi(T y, y) \leq \phi(T x, y)+\phi(T y, x),
   \end{align}
which implies that
   \begin{align}
     \sigma\phi(u, Tx) \leq \phi(u, x).
   \end{align}
Hence, $\sigma^{n}\phi(u, T^{n}x) \leq \phi(u, x)$. Then $\|u - T^{n}x\| \to 0$ follows from $\phi(u, T^{n}x)\to 0$. If $X$ is $2$-uniformly convex, then there is a constant $\mu$ such that $\frac{1}{\mu}\|u-T^n x\|^2 \leq \phi(u,T^n x) $ by Lemma \ref{phiinequ1}, therefore,  $\|u - T^{n}x\|^2 \leq \frac{\mu}{\sigma^n} \phi(u,x)$.
\end{proof}

According to Proposition \ref{f and J}, the conclusion of Theorem \ref{sc of T} remains valid for the resolvent \( J_{r}\). Consequenctly, the result establishes the convergence of the proximal point algorithm for monotone operators in Banach spaces.
\begin{theorem}\label{proximal point}
Let \( X \) be a strictly convex and smooth and reflexive Banach space, let \( C \) be a nonempty closed convex subset of \( X \), and let \( A \subset X \times X^* \) be a $\alpha$-monotone operator ($\alpha >0$) satisfying $D(A) \subset C \subset J^{-1}R(J + rA)$ for all \( r > 0 \). Let \( r \) be a positive real number and let \( J_r x = (J + rA)^{-1} Jx \) for all \( x \in C \). Then $\{J_{r}^{n}x\}$ converge strongly to fixed point of $A^{-1}0$ for any $x\in C$. If $X$ is $2$-uniformly convex, then there is a constant $\mu$ such that $\|u - J_{r}^{n}x\|^2 \leq \frac{\mu}{\sigma^n} \phi(u,x)$, $u\in A^{-1}0$ and $\sigma = 1+r\alpha$.
\end{theorem}

\begin{remark}
    There is a well-known weak convergence result of the proximal point algorithm for monotone operators in Banach spaces (see, for example, \cite{KT2008}), which can be stated as follows.

    Let \( X \) be a uniformly convex Banach space whose norm is uniformly G\^{a}teaux differentiable, and let \( A \subset X \times X^* \) be a monotone operator such that \( A^{-1}0 \) is nonempty. Let \( C \) be a nonempty closed convex subset of \( X \) satisfying $D(A) \subset C \subset J^{-1}R(J + rA)$ for all \( r > 0 \). Let \( r \) be a positive real number and let \( J_r x = (J + rA)^{-1} J x \)  for all \( x \in C \). If \( J \) is weakly sequentially continuous, then for every \( x \in C \), \( \{ J_r^n x \} \) converges weakly to an element of \( A^{-1}0 \).
\end{remark}

   Even if \( A \) is an \( \alpha \)-monotone operator ($\alpha > 0$) in the sense of Definition \ref{classical def}, a convergence result analogous to that in Theorem \ref{proximal point} cannot, in general, be obtained. In contrast, by combining Theorem \ref{dense2} and Theorem  \ref{proximal point}, we derive the following corollary.

\begin{corollary}
    Let \( X \) be a strictly convex and smooth and reflexive Banach space, let \( C \) be a nonempty closed convex subset of \( X \), and let \( A \subset X \times X^* \) be a maximal $\alpha$-monotone operator ($\alpha >0$) under the Definition \ref{classical def}, which satisfying $D(A) \subset C \subset J^{-1}R(J + rA)$ for all \( r > 0 \). If $A^{-1}0$ is nonempty. Let $A_n = A + \frac{1}{\sqrt{n}}J$, then
    $\{J^{n}_{rA_n}x\}$ converges strongly to fixed point of $A^{-1}0$ for any $x\in C$. If $X$ is $2$-uniformly convex, then there is a constant $\mu$ such that $\|u_n - J_{rA_n}^{n}x\|^2 \leq \frac{\mu}{\sigma^n} \phi(u,x)$, $u_n\in A_{n}^{-1}0$, $u\in A^{-1}0$, $\sigma = 1+\frac{r}{\sqrt{n}}$, and $\|u_n - u\| \leq \frac{1}{\alpha\sqrt{n}}\|u\|$.
\end{corollary}

\begin{remark}
    Let \( u \in A^{-1}0 \). For any given \( \epsilon > 0 \) and any \( x \in D(A) \), we can easily compute and select \( n_0 \) such that
    \[
    \frac{\mu}{\sigma^{n_0}} \phi(u,x) + \frac{1}{\alpha \sqrt{n_0}} \|u\| \leq \epsilon \|u\|,
    \]
    which implies that \( \|J^{n_0}_{rA_{n_0}} - u\| \leq \epsilon \|u\| \). Additionally, if we know that the upper bound of \(\|u\|\) is \( M \), we can easily choose \( n_1 \) satisfying
    \[
    \|J^{n_1}_{rA_{n_1}} - u\| \leq \epsilon.
    \]
\end{remark}

\section{R-linear convergence of the forward-reflected-backward splitting in Banach spaces}

In 2020, Malitsky and Tam \cite{MT1} proposed the so-called forward-reflected-backward splitting algorithm for solving the monotone inclusion problem: find $x\in H$ such that $0\in Ax+Bx$, where $A:H\rightarrow 2^{H}$ is maximally monotone and $B:H\rightarrow H$ is monotone and $L$-Lipschitz continuous for some $L>0$. The forward-reflected-backward splitting algorithm is defined as follows:
$$
x_{n+1} = J_{\lambda_n A}( x_n - \lambda_n Bx_n - \lambda_{n-1}(Bx_n -
Bx_{n-1}) ),
$$
where $\{\lambda_n\}\subseteq \left[\epsilon, \tfrac{1-2\epsilon}{2L}\right]$ for some $\epsilon >0$. They not only proved the weak convergence of the algorithm under general conditions but also established its strong convergence under the strong monotonicity assumption, and rigorously demonstrated that the algorithm achieves an R-linear rate of convergence. Furthermore, Bello et al. \cite{BCIO} extended the forward-reflected-backward splitting algorithm from Hilbert spaces $H$ to general Banach spaces $X$, and proved weak convergence of the generated sequence in real 2-uniformly convex and uniformly smooth Banach spaces. However, in general Banach spaces, the strong convergence and R-linear convergence of the forward-reflected-backward splitting algorithm have not been established.
The main difficulties arises from the inadequacy of the standard definition of strong monotonicity for operators in Banach space frameworks. In this section, we establish strong convergence with an $R$-linear rate for the forward-reflected-backward splitting algorithm in real 2-uniformly convex and uniformly smooth Banach spaces under Definition \ref{our def}.

\begin{theorem}
Let $X$ be a real 2-uniformly convex and uniformly smooth Banach space. Let \(A: X \to 2^{X^{*}}\) be a maximally monotone and $\alpha$-monotone operator and \(B: X \to X^{*}\) be $\beta$-monotone and Lipschitz, $\alpha+\beta>0$. For \(x_{0}, x_{-1} \in X\) defined sequence \(x_{n}\) by
  \begin{align}\label{frb}
    x_{n+1}=J_{\lambda_{n}}^{A} \circ J^{-1}\left(J x_{n}-\lambda_{n} B x_{n}-\lambda_{n-1}\left(B x_{n}-B x_{n-1}\right)\right),~ n \geq 0,
  \end{align}
where \(\lambda_{n} \subseteq[\epsilon, \frac{1-2 \epsilon}{2 \mu L}]\) for some \(\alpha+\beta>\epsilon>0\) and \(\mu \geq 1\) . Suppose \((A+B)^{-1}(0) \neq \emptyset\), then the sequence \({x_n}\) converges strongly to an element of \((A+B)^{-1}(0)\) with R-linear rate.
\end{theorem}

\begin{proof}
  Let \(x^{*} \in(A+B)^{-1}(0)\) . So that
  \begin{align}\label{Ax*}
    -B x^{*} \in A x^{*}.
  \end{align}
From (\ref{frb}) we have that
\begin{align}\label{Axn+1}
  \frac{1}{\lambda_{n}}\left(J x_{n}-\lambda_{n} B x_{n}-\lambda_{n-1}\left(B x_{n}-B x_{n-1}\right)-J x_{n+1}\right) \in A x_{n+1}
\end{align}
Using (\ref{Ax*}), (\ref{Axn+1}) and the strong monotonicity of $A$, we obtain
\begin{equation}\label{A m}
\begin{aligned}
  &\left< J x_{n+1}-J x_{n}+\lambda_{n}\left(B x_{n}-B x^{*}\right)+\lambda_{n-1}\left(B x_{n}-B x_{n-1}\right), x^{*}-x_{n+1}\right>\\
  \geq& \alpha\langle x^{*}-x_{n+1}, Jx^{*}-Jx_{n+1}\rangle
\end{aligned}
\end{equation}
By Lemma \ref{phiequ}(1), we have
\begin{align}\label{J 1}
  2\left< J x_{n+1}-J x_{n}, x^{*}-x_{n+1}\right>=\phi\left(x^{*}, x_{n}\right)-\phi\left(x^{*}, x_{n+1}\right)-\phi\left(x_{n+1}, x_{n}\right).
\end{align}
Also
\begin{equation}\label{B 1}
  \begin{aligned} & \left< B x_{n}-B x^{*}, x^{*}-x_{n+1}\right> \\
  =&\left< B x_{n+1}-B x^{*}, x^{*}-x_{n+1}\right>+\left< B x_{n}-B x_{n+1}, x^{*}-x_{n+1}\right>.
  \end{aligned}
\end{equation}
And
\begin{equation}\label{B 2}
  \begin{aligned} & \left< B x_{n}-B x_{n-1}, x^{*}-x_{n+1}\right> \\
   =&\left< B x_{n}-B x_{n-1}, x^{*}-x_{n}\right>+\left< B x_{n}-B x_{n-1}, x_{n}-x_{n+1}\right>.
  \end{aligned}
\end{equation}
Notice that $B$ is $\beta$-monotone, we have
\begin{align}\label{B monotone}
  \left< B x_{n+1}-B x^{*}, x^{*}-x_{n+1}\right> \leq -\beta\left< x^{*}-x_{n+1}, Jx^{*}-Jx_{n+1}\right>.
\end{align}
\begin{equation}\label{J 2}
  \begin{aligned}
  (\alpha+\beta)\langle x^{*}-x_{n+1}, Jx^{*}-Jx_{n+1}\rangle &= \frac{\alpha+\beta}{2}\phi(x^*, x_{n+1}) + \frac{\alpha+\beta}{2}\phi(x^*, x_{n+1})\\
  &\geq \frac{\alpha+\beta}{2}\phi(x^*, x_{n+1}).
\end{aligned}
\end{equation}

Substituting (\ref{J 1}), (\ref{B 1}) and (\ref{B 2}) in (\ref{A m}), and since (\ref{B monotone}) (\ref{J 2}), we have:
\begin{equation}\label{frb ine1}
  \begin{aligned}
  & (1+\frac{\alpha+\beta}{2})\phi\left(x^{*}, x_{n+1}\right)+2 \lambda_{n}\left< B x_{n+1}-B x_{n}, x^{*}-x_{n+1}\right> \\
  \leq &\phi\left(x^{*}, x_{n}\right)+2 \lambda_{n-1}\left< B x_{n}-B x_{n-1}, x^{*}-x_{n}\right> \\
  &+2 \lambda_{n-1}\left< B x_{n}-B x_{n-1}, x_{n}-x_{n+1}\right>-\phi\left(x_{n+1}, x_{n}\right).
  \end{aligned}
\end{equation}
Rearranging the equation we have
\begin{equation}\label{frb ine2}
  \begin{aligned}
  & (1+\frac{\alpha+\beta}{2})\phi\left(x^{*}, x_{n+1}\right)+2 \lambda_{n}\left< B x_{n+1}-B x_{n}, x^{*}-x_{n+1}\right>+\phi\left(x_{n+1}, x_{n}\right)\\
  \leq &\phi\left(x^{*}, x_{n}\right)
   +2 \lambda_{n-1}\left< B x_{n}-B x_{n-1}, x^{*}-x_{n}\right>
   +2 \lambda_{n-1}\left< B x_{n}-B x_{n-1}, x_{n}-x_{n+1}\right>
   \end{aligned}
\end{equation}
Using the Lipschitz property of $B$ and Lemma \ref{phiinequ1} we have
\begin{equation}\label{B 3}
  \begin{aligned}
  & 2 \lambda_{n-1}\left< B x_{n}-B x_{n-1}, x_{n}-x_{n+1}\right> \\
  \leq & 2 \lambda_{n-1}\left\| B x_{n}-B x_{n-1}\right\| \left\| x_{n}-x_{n+1}\right\| \\
  \leq & 2 \lambda_{n-1}L\left\| x_{n}-x_{n-1}\right\| \left\| x_{n}-x_{n+1}\right\| \\
  \leq & \lambda_{n-1} L\left(\left\| x_{n}-x_{n-1}\right\| ^{2}+\left\| x_{n}-x_{n+1}\right\| ^{2}\right) \\
  \leq & \lambda_{n-1} \mu L\left(\phi\left(x_{n}, x_{n-1}\right)+\phi\left(x_{n+1}, x_{n}\right)\right).
  \end{aligned}
\end{equation}
Substituting (\ref{B 3}) in (\ref{frb ine2}) we have
\begin{equation}\label{frb ine3}
  \begin{aligned}
  & (1+\frac{\alpha+\beta}{2})\phi\left(x^{*}, x_{n+1}\right)+2 \lambda_{n}\left< B x_{n+1}-B x_{n}, x^{*}-x_{n+1}\right> \\
  & \quad+\left(1-\lambda_{n-1} \mu L\right) \phi\left(x_{n+1}, x_{n}\right) \leq \phi\left(x^{*}, x_{n}\right) \\
  & \quad+2 \lambda_{n-1}\left< B x_{n}-B x_{n-1}, x^{*}-x_{n}\right>+\lambda_{n-1} \mu L \phi\left(x_n, x_{n-1}\right) . \end{aligned}
\end{equation}
Now, \(\lambda_{n} \in[\epsilon, \frac{1-2 \epsilon}{2 \mu L}]\) implies \(\lambda_{n} \mu L<\frac{1}{2}\) and \(\frac{1}{2}+\epsilon \leq 1-\lambda_{n} \mu L\) . So that (\ref{frb ine3}) becomes
\begin{equation}\label{frb ine4}
  \begin{aligned}
  & (1+\frac{\alpha+\beta}{2})\phi\left(x^{*}, x_{n+1}\right)+2 \lambda_{n}\left< B x_{n+1}-B x_{n}, x^{*}-x_{n+1}\right> \\
  & \quad+\left(\frac{1}{2}+\epsilon\right) \phi\left(x_{n+1}, x_{n}\right) \leq \phi\left(x^{*}, x_{n}\right) \\
  & \quad+2 \lambda_{n-1}\left< B x_{n}-B x_{n-1}, x^{*}-x_{n}\right>+\frac{1}{2} \phi\left(x_n, x_{n-1}\right)
  \end{aligned}
\end{equation}
Let \( a_{n+1} = \frac{1}{2}\phi(x^*, x_{n+1}) \), \( b_{n+1} = \frac{1}{2}\phi(x^*, x_{n+1}) + 2\lambda_n \langle Bx_{n+1} - Bx_n, x^* - x_{n+1} \rangle + \frac{1}{2}\phi(x_{n+1}, x_n) \).
Then we have
\[
b_{n+1} \geq \left( \frac{1}{2} - \lambda_n \mu L \right)\phi(x^*, x_{n+1}) + \left( \frac{1}{2} - \lambda_n \mu L \right)\phi(x_{n+1}, x_n)\geq 0.
\]
And
(\ref{frb ine4}) is equivalent to
\begin{align}\label{frb ine5}
  (1+\alpha+\beta)a_{n+1} + b_{n+1} + \varepsilon \phi(x_{n+1}, x_n) \leq a_n + b_n.
\end{align}
Since
\begin{equation}\label{bn+1}
\begin{aligned}
b_{n+1} \leq& \left( \frac{1}{2} + \lambda_n \mu L \right)\phi(x^*, x_{n+1}) + \left( \frac{1}{2} + \lambda_n \mu L \right)\phi(x_{n+1}, x_n) \\
=& (1 + 2\lambda_n \mu L) \cdot \frac{1}{2}\phi(x^*, x_{n+1}) + \left( \frac{1}{2} + \lambda_n \mu L \right)\phi(x_{n+1}, x_n) \\
\leq& (1 + 2\lambda_n \mu L)a_{n+1} + \phi(x_{n+1}, x_n).
\end{aligned}
\end{equation}
Then by (\ref{frb ine5}) and (\ref{bn+1}), we have
\begin{equation}\label{frb ine6}
  \begin{aligned}
  &(1+\alpha+\beta)a_{n+1} + b_{n+1} + \varepsilon \phi(x_{n+1}, x_n)\\
  =& (1+\alpha+\beta)a_{n+1} + (1+\frac{\epsilon}{2})b_{n+1} - \frac{\epsilon}{2}b_{n+1} + \varepsilon \phi(x_{n+1}, x_n)\\
  \geq & (1+\alpha+\beta - \frac{\epsilon}{2}(1 + 2\lambda_n \mu L))a_{n+1} + (1+\frac{\epsilon}{2})b_{n+1}\\
  \geq & (1+\alpha+\beta -\epsilon)a_{n+1} + (1+\frac{\epsilon}{2})b_{n+1}.
\end{aligned}
\end{equation}
Then, let
$\theta = \min\{1+\alpha+\beta -\epsilon, 1+\frac{\epsilon}{2}\} > 1$,
then, by (\ref{frb ine6})
\begin{align}
  \theta (a_{n+1} + b_{n+1}) \leq a_{n} + b_{n},
\end{align}
which establishes that $x_n \to x^*$ with $\|x_{n+1} - x^*\|^2 \leq \frac{\mu}{\theta^n}(a_1 + b_1)$ by Lemma \ref{phiinequ1}.
\end{proof}

Let $\alpha>0$, and $\beta=0$, we can get the following corollary.

\begin{corollary}\label{2operators}
Let $X$ be a real 2-uniformly convex and uniformly smooth Banach space. Let \(A: X \to 2^{X^{*}}\) be a maximally monotone and $\alpha$-strong monotone operator and \(B: X \to X^{*}\) be monotone and Lipschitz. For \(x_{0}, x_{-1} \in X\) defined sequence \(x_{n}\) by
  \begin{align}\label{frb2}
    x_{n+1}=J_{\lambda_{n}}^{A} \circ J^{-1}\left(J x_{n}-\lambda_{n} B x_{n}-\lambda_{n-1}\left(B x_{n}-B x_{n-1}\right)\right),~ n \geq 0,
  \end{align}
where \(\lambda_{n} \subseteq[\epsilon, \frac{1-2 \epsilon}{2 \mu L}]\) for some \(\alpha>\epsilon>0\) and \(\mu \geq 1\) . Suppose \((A+B)^{-1}(0) \neq \emptyset\), then the sequence \({x_n}\) converges strongly to an element $x^*$ of \((A+B)^{-1}(0)\) with $\|x_{n+1} - x^*\|^2 \leq \frac{M}{\theta^n}$ for some constant $ \theta>1$ and $M$.
\end{corollary}


\begin{remark}
Corollary 4.2 extends Theorem 2.9 in \cite{MT1} from Hilbert spaces to Banach spaces, which constitutes one of the main contributions of this paper.
\end{remark}

\section{Conclusions}

In Hilbert spaces, the properties of the resolvent operator associated with $\alpha$-monotone operators have been extensively studied. In general Banach spaces, based on Definition \ref{classical def}, when $\alpha = 0$, the resolvent of the corresponding monotone operator has been proven to be firmly nonexpansive type. However, when $\alpha \neq 0$, it becomes considerably more difficult to establish the properties of the resolvent operator in Banach spaces directly from the $\alpha$-monotonicity given by Definition \ref{classical def}. To address this issue, this paper introduced a new definition of $\alpha$-monotonicity in smooth Banach spaces (Definition \ref{our def}). We compared and analyzed the relationship between this new definition and Definition \ref{classical def}, and then systematically proved the properties of the resolvent operator under the new definition. These properties are consistent with the existing results in the Hilbert space setting, thereby demonstrating the reasonableness and generalizability of the new definition. Furthermore, we employed the newly proposed $\alpha$-monotonicity assumption to operator splitting methods and proved that the forward-reflected-backward splitting algorithm, under suitable parameter conditions, not only achieves strong convergence but also attains $R$-linear convergence. This provides a new theoretical foundation for the study of $\alpha$-monotone operators and the convergence analysis of related numerical algorithms in general Banach spaces.

In future work, the proposed definition of $\alpha$-monotonicity can be employed to investigate the strong convergence and R-linear convergence rate of other operator splitting algorithms in general Banach spaces, such as the forward-backward-half forward splitting algorithm \cite{BAD2018}, the semi-reflected-forward-backward splitting algorithm \cite{CV2019}, and the outer reflected forward-backward splitting algorithm \cite{YZT2020}, etc.

%

\section*{Author contributions}

All authors contributed to the study conception.

\section*{Data availability statement}

No datasets were generated or analysed during the current study.

\section*{Competing interests}

The authors declare no competing interests.

\bibliographystyle{unsrt}

\end{document}